\documentclass{article}
\usepackage[pdftex]{hyperref}
\hypersetup{
  colorlinks   = true, 
  urlcolor     = gray, 
  linkcolor    = red, 
  citecolor   =  blue
}
\usepackage{amsmath,amsfonts,amsthm,amssymb,graphicx,xcolor}
\usepackage[all,2cell,ps]{xy}

\theoremstyle{plain}
\newtheorem{theorem}[equation]{Theorem}

\newtheorem{lemma}[equation]{Lemma}
\newtheorem{proposition}[equation]{Proposition}

\theoremstyle{definition}

\newtheorem{definitions}[equation]{Definition}

\newtheorem{remark}[equation]{Remark}

\newtheorem{question}[equation]{Question}

\newcommand{\IR}{\mathbb{R}}

\DeclareMathOperator{\csch}{csch}

\title{Price Inequality and the Growth of Harmonic Functions on Non-Positively Curved Manifolds}
 \author{\small{Luca F. Di Cerbo}\footnote{Supported in part by NSF grant DMS-2104662} \\ \scriptsize{University of Florida} \\ \footnotesize{\textsf{ldicerbo@ufl.edu}} \and \small{Hayden Hunter}\\ \scriptsize{Florida International University} \\ \footnotesize{\textsf{hhunter@fiu.edu}} \and
\small{Aaron K. Thrasher} \\ 
 \footnotesize{\textsf{akthrasher17291@gmail.com}}}
\date{}

\begin{document}

\maketitle

\begin{abstract}
We obtain effective estimates for the growth rate of the $L^2$-energy of harmonic functions on geodesic balls in complete simply connected non-positively curved Riemannian manifolds with pinched sectional curvature. Our study relies upon a double-sided Price inequality for harmonic functions. Finally, we apply this circle of ideas to study the analytical structure of a potential counterexample to the Singer conjecture in degree one. 
\footnote{2020 {\em Mathematics Subject Classification}. Primary 53C21; Secondary 31B05.}
\end{abstract}

\vspace{9cm}

\tableofcontents

\vspace{1cm}


\section{Introduction and Organization}
Let $(M^n, g)$ be a non-compact complete Riemannian manifold of infinite volume. In this paper, we study the behavior of
\begin{align}\label{Gf}
g_{f}(R):=\int_{B_R(p)}f^2d\mu_g, 
\end{align}
as $R\to \infty$, where $f:M\to \IR$ is a harmonic function. In this context, recall the following result of S.-T. Yau (cf. \cite[Theorem 3]{Yau76}):
\begin{theorem}[Yau]
Let $f$ be a harmonic function on complete Riemannian manifold $M^{n}$. Then, $\int_{M}f^{p} = \infty$ for $p >1$ or $f$ is a constant function. 
\end{theorem}
Thus, if $f$ is non-constant (or if $Vol_{g}(M)=\infty$) Yau's result implies that
\[
\lim_{R\to\infty}g_{f}(R)=\infty.
\]
Therefore, we aim to obtain \emph{effective} estimates for the divergence rate of the $L^{2}$-energy of harmonic functions on geodesic balls in $M$. We make our study by developing a \emph{double-sided Price inequality}  for harmonic functions on complete simply connected Riemannian manifolds of non-positive sectional curvature. Our study is a natural extension of the theory of Price inequality for harmonic forms developed by the first-named author and Stern in \cite{DS22}, and later extended in various different geometric and analytic contexts \cite{DS25a}, \cite{DS25b}, \cite{DH25}, and \cite{Ste25}. Recall that the so-called Price inequality first appeared in the context of Yang-Mills theory in \cite{Pri83}.

The typical approach in studying the growth rate of the $L^2$-energy of harmonic functions on geodesic balls is via the mean-value property of sub-harmonic functions. Suppose that $(M, g)$ is a Riemannian manifold with $\sec_{g}\leq k$, and let $B_{R}(p)$ be any geodesic ball. Given a harmonic function $f: M\to \IR$ we have that $u:=f^2$ is sub-harmonic and non-negative. Thus, given a point $p\in M$ where $f(p)\neq 0$, the mean-value property for sub-harmonic functions tells us that 
\begin{align}\label{meanvalue}
\int_{B_R(p)}f^2d\mu_g\geq f^2(p) V_{k}(R)
\end{align}
where $V_{k}(R)$ is the volume of a geodesic ball of radius $R$ in the space form of constant curvature $k$. For more details, we refer to Schoen-Yau lectures \cite[Chapter 2, Section 6]{SchoenY}.

In this paper, we develop a double-sided Price inequality bound for the $L^2$-energy of a harmonic function as defined in Equation \eqref{Gf}. Section \ref{PriceI} is devoted to the development of this estimate. The lower bound we obtain improves on the bound given in Equation \eqref{meanvalue} by incorporating  an auxiliary positive function $\mu$ associated to the given harmonic function. For the details about these estimates, we refer to Definition \ref{defnmu} and Theorem \ref{doublebound}. Interestingly, the $\mu$-function introduced in Definition \ref{defnmu} enjoys the same monotonicity properties of Almgren frequency function when restricted to Euclidean harmonic functions! This interesting fact is outlined in  Appendix \ref{Aaron}. In this regard, this $\mu$-function could be thought of as an analogue of Almgren frequency function in the context of general non-positively curved Riemannian manifolds. Unfortunately, as discussed in \cite[Chapter 4]{Aaron}, the $\mu$-function is not monotonic already on certain harmonic functions on real-hyperbolic space, see more precisely \cite[Proposition 4.2.2]{Aaron}.

In Section \ref{contradiction}, we explore the implications of our estimates in the context of the Singer conjecture for negatively curved closed Riemannian manifold. More precisely, we explore the analytic nature of a possible counterexample to such conjecture in degree one (1st $L^2$-Betti number). The Singer conjecture is currently still open even in degree one on negatively curved spaces in dimensions $n\geq 4$, \cite[Section 7]{DS22}.

In Section \ref{PH}, we explore the consequences of our double-sided Price inequality for \textit{positive} harmonic functions on simply connected negatively curved Riemannian manifolds with pinched sectional curvature. We refer to Theorem \ref{PHmain} for the precise estimates, and we note that the upper bound in this theorem relies on a bound on the integral of the hyperbolic Poisson kernel over geodesic spheres. The latter result is outlined in Appendix \ref{Hayden}.\\

\noindent\textbf{Acknowledgments}. LFDC  thanks Mark Stern for numerous discussions related to the contents of this paper. AKT thanks Thalia Jeffres and Xiaolong Li for the invitation to present these results at a Special Session in the AMS 2025 Spring Central Sectional Meeting.

\section{Double-Sided Price Inequality}\label{PriceI}

 In this section, we derive a double-sided Price inequality for the $L^2$-norm of a harmonic function $f$ on geodesic spheres in a complete Riemannian manifold $(M^n, g)$ of non-positive sectional curvature.

\begin{definitions}\label{defnmu}
For $f: M \to \IR$ harmonic, we define the function
$$\mu_{f, p}(R) = \frac{\int_{B_{R}(p)}\partial_{r}f^2 d\mu_{g}}{\int_{S_{R}(p)}f^2 d\sigma_{R}} = \frac{2\int_{0}^{R}\int_{B_{r}(p)}|\nabla f|^{2}d\mu_{g}dr}{\int_{S_{R}(p)}f^{2}d\sigma_{R}}.$$
When $f$ and $p$ are understood, we suppress the definition above to $\mu(R)$.
\end{definitions}

\begin{remark}\label{lessone}
Since $f$ is harmonic, $\mu$ is clearly well-defined and non-negative. Moreover, $\mu(R)$ is strictly positive for $R>0$ unless $f$ is a constant. Finally, it follows from the divergence theorem that $0\leq \mu < 1$ for spaces whose geodesic spheres have positive mean curvature, see Equation \eqref{eq:4.1}. Spaces with non-positive sectional curvature satisfy this requirement because of Rauch's comparison theorem, see for example \cite[Theorem 6.4.3]{Petersen}.
\end{remark}

We can now state and prove our main estimate.

\begin{theorem}[Double-Sided Price Inequality]\label{doublebound}
\label{thm:4.1}
Let $(M^{n},g)$ be complete and simply connected with sectional curvature $k' \leq \text{sec}_{g} \leq k \leq 0$. Let $f: M \to \IR$ be non-constant and harmonic. Then, over the geodesic ball $B_{R}$ centered at $p$ with $R\geq 1$ we have the estimate
$$C_1e^{\int_{1}^{R}\frac{H_{k}(s)}{1-\mu(s)}ds} \leq \int_{B_{R}}f^{2}d\mu_{g} \leq C_2e^{\int_{1}^{R}\frac{H_{k'}(s)}{1-\mu(s)}ds}$$
where $C_1$ and $C_2$ are some positive constants (depending on $p$ and $f$), $H_{k}$ is mean curvature of geodesic spheres in the $n$-dimensional manifold with constant sectional curvature $k$, and $\mu$ is as defined in \textcolor{blue}{\ref{defnmu}}.
\end{theorem}
\begin{proof}
Let $(M^{n},g)$ be complete and simply connected with sectional curvatures $k' \leq \text{sec}_{g} \leq k < 0$. Let $f: M \to \IR$ be non-constant and harmonic. Distinguish a point $p \in M$ as the center of our geodesic spherical coordinate system and consider the vector field $f^2\partial_{r}$, noting that it is well-defined everywhere except $p$ and bounded around this point. Consider the following:
$$\text{div}(f^{2}\partial_{r}) = f^{2}\text{div}(\partial_{r}) + g(\nabla f^{2}, \partial_{r}) =  f^{2}H + \partial_{r}f^{2},$$
where 
\[
H(r, \theta): S_{r}(p)\longrightarrow\IR 
\]
is the mean curvature function over the geodesic sphere $S_{r}(p)$. Integrate both sides:
$$\int_{B_{R}}\text{div}(f^{2}\partial_{r}) d\mu_{g} = \int_{B_{R}}Hf^{2}d\mu_{g} + \int_{B_{R}}\partial_{r}f^2 d\mu_{g}.$$
By the divergence theorem, we retrieve
\begin{equation}\label{eq:4.1}
\int_{S_{R}}f^{2}d\sigma_{R} = \int_{B_{R}}Hf^{2}d\mu_{g} + \int_{B_{R}}\partial_{r}f^2 d\mu_{g}.
\end{equation}
By Green's identity, we have
\begin{align*}
\int_{B_{R}}\Delta_{g}f^{2}d\mu_{g}=2\int_{B_{R}}|\nabla f|^{2}d\mu_{g}=\int_{S_{R}}\partial_{r}f^{2}d\sigma_{R}.
\end{align*}
Now, it is clear that $2\int_{B_{R}}|\nabla f|^{2}d\mu_{R} > 0$ for $f$ non-constant. Therefore, we introduce an auxiliary function $\mu(R)$ as in Definition \ref{defnmu}:
\begin{equation}\label{eq:4.2}
\int_{B_{R}}\partial_{r}f^{2}d\mu_{g} = \mu(R)\int_{S_{R}}f^{2}d\sigma_{R}.
\end{equation}
Applying this auxiliary equation to Equation \eqref{eq:4.1}, we receive
$$(1 - \mu(R))\int_{S_{R}}f^{2}d \sigma_{R} = \int_{B_{R}}Hf^{2}d \mu_{g}.$$
Rauch's comparison theorem (\cite[Theorem 6.4.3]{Petersen}) ensures that $H>0$ for any $R$ so that $\mu(R)<1$ (\textit{cf}. Remark \ref{lessone}).  We now multiply through by a function $\phi_{k}'(r)$ and integrate with respect to $r$ from $1$ to $R$:
$$\int_{1}^{R}\phi_{k}'(r)(1 - \mu(r))\int_{S_{r}}f^{2}d \sigma_{r}dr = \int_{1}^{R}\phi_{k}'(r)\int_{B_{r}}Hf^{2}d \mu_{g}dr.$$
Next, we integrate by parts:
$$\int_{1}^{R}\phi_{k}'(r)\int_{B_{r}}Hf^{2}d \mu_{g}dr = \left[\phi_{k}(r)\int_{B_{r}}Hf^{2}d\mu_{g}\right]_{1}^{R} - \int_{1}^{R}\phi_{k}(r)\int_{S_{r}}Hf^{2}d\sigma_{r}dr$$
$$= \phi_{k}(R)\int_{B_{R}}Hf^{2}d\mu_{g} - \phi_{k}(1)\int_{B_{1}}Hf^{2}d\mu_{g} - \int_{1}^{R}\phi_{k}(r)\int_{S_{r}}Hf^{2}d\sigma_{r}dr.$$
By combining the equations above, we receive
\begin{align*}
\int_{1}^{R}\left[\phi_{k}'(r)(1 - \mu(r))\int_{S_{r}}f^{2}d \sigma_{r} + \phi_{k}(r)\int_{S_{r}}Hf^{2}d\sigma_{r}\right]dr \\
 = \phi_{k}(R)\int_{B_{R}}Hf^{2}d\mu_{g} - \phi_{k}(1)\int_{B_{1}}Hf^{2}d\mu_{g}.
\end{align*}
We simplify the left-hand side for clarity:
$$\int_{B_{R}\setminus B_{1}}f^{2}[\phi_{k}'(r)(1 - \mu(r)) + \phi_{k}(r)H]d\mu_{g} = \phi_{k}(R)\int_{B_{R}}Hf^{2}d\mu_{g} - \phi_{k}(1)\int_{B_{1}}Hf^{2}d\mu_{g}.$$
Rauch's comparison theorem (\cite[Theorem 6.4.3]{Petersen}), assuming $\phi_{k}(r) > 0$, implies that
\begin{align}\label{eq:4.3}
\phi_{k}(R)\int_{B_{R}}Hf^{2}d\mu_{g} -& \phi_{k}(1)\int_{B_{1}}Hf^2d\mu_{g} \\ \notag 
&\geq \int_{B_{R}\setminus B_{1}}f^{2}[\phi_{k}'(r)(1-\mu(r)) + H_{k}(r)\phi_{k}(r)]d\mu_{g}.
\end{align}
We now select $\phi_{k}(r)$ so that
\begin{center}
    $\begin{cases}
        \phi_{k}'(r)(1-\mu(r)) + H_{k}(r)\phi_{k}(r) = 0\\
        \phi_{k}(1) = 1
    \end{cases}$
\end{center}
is satisfied. We proceed to solve this initial value problem by separation of variables:
\begin{align*}
    \frac{d \phi_{k}}{ds}(1-\mu(s)) + H_{k}(s)\phi_{k}(s) &= 0\\
    \frac{1}{\phi_{k}}d\phi_{k} &= -\frac{H_{k}(s)}{1-\mu(s)}\\
    \phi_{k}(R) &= e^{-\int_{1}^{R}\frac{H_{k}(s)}{1-\mu(s)}ds}>0.
\end{align*}
Applying this result to Equation \eqref{eq:4.3}, we receive a lower \textbf{Price} bound,
$$\phi_{k}(R)\int_{B_{R}}Hf^{2}d\mu_{g} \geq C = \int_{B_{1}} Hf^2 d\mu_g > 0.$$
That is to say, we have the inequality
$$\int_{B_{R}}Hf^{2}d\mu_{g} \geq Ce^{\int_{1}^{R}\frac{H_{k}(s)}{1-\mu(s)}ds}.$$
We now want to develop the upper bound. To do this, we return to
$$(1 - \mu(R))\int_{S_{R}}f^{2}d \sigma_{R} = \int_{B_{R}}Hf^{2}d \mu_{g}.$$
Multiply through by \textit{another} choice $\phi'_{k'}(r)$ and integrate from $1$ to $R$:
$$\int_{1}^{R}\phi_{k'}'(r)(1 - \mu(r))\int_{S_{r}}f^{2}d \sigma_{r}dr = \int_{1}^{R}\phi_{k'}'(r)\int_{B_{r}}Hf^{2}d \mu_{g}dr.$$
Proceeding exactly as before, we obtain:
$$\int_{B_{R}\setminus B_{1}}f^{2}[\phi_{k'}'(r)(1 - \mu(r)) + \phi_{k'}(r)H]d\mu_{g} = \phi_{k'}(R)\int_{B_{R}}Hf^{2}d\mu_{g} - \phi_{k'}(1)\int_{B_{1}}Hf^{2}d\mu_{g}.$$
By using the upper estimate on  $H$ given by Rauch's theorem  (\cite[Theorem 6.4.3]{Petersen}), and assuming $\phi_{k'}(r) > 0$, we receive
\begin{align} \label{eq:4.3(2)}
 \int_{B_{R}\setminus B_{1}}f^{2}[\phi_{k'}'(r)(1-\mu(r)) +& H_{k'}(r)\phi_{k'}(r)]d\mu_{g}\\ \notag 
&\geq   \phi_{k'}(R)\int_{B_{R}}Hf^{2}d\mu_{g} - \phi_{k'}(1)\int_{B_{1}}Hf^2d\mu_{g}.
\end{align}
We select $\phi_{k'}(r)$ so that
\begin{center}
    $\begin{cases}
        \phi_{k'}'(r)(1-\mu(r)) + H_{k^\prime}(r)\phi_{k'}(r) = 0\\
        \phi_{k'}(1) = 1.
    \end{cases}$
\end{center}
is satisfied, so that
\begin{align*}
    \phi_{k'}(r) &= e^{-\int_{1}^{r}\frac{H_{k'}(s)}{1-\mu(s)}ds}>0.
\end{align*}
Applying this result to Equation \eqref{eq:4.3(2)}, we receive an \emph{upper} \textbf{Price} bound,
$$\phi_{k'}(R)\int_{B_{R}}Hf^{2}d\mu_{g} \leq C.$$
This gives us the double-sided bound:
\begin{align}\label{doubleHf}
Ce^{\int_{1}^{R}\frac{H_{k}(s)}{1-\mu(s)}ds} \leq \int_{B_{R}}Hf^{2}d\mu_{g} \leq Ce^{\int_{1}^{R}\frac{H_{k'}(s)}{1-\mu(s)}ds}.
\end{align}
Hence, we have the desired \emph{exponential} bounds on $\int_{B_{R}}Hf^{2}d\mu_{g}$. We would like to compare $\int_{B_{R}}Hf^{2}d\mu_{g}$ with $\int_{B_{R}}f^{2}d\mu_{g}$ to retrieve the same bounds for this second integral. If $k' \leq \text{sec}_{g} \leq k < 0$, we have the following.
First, we apply the Rauch's comparison theorem twice:
$$\frac{\int_{B_{R}}f^{2}d\mu_{g}}{\int_{B_{R}}H_{k^\prime}(r)f^{2}d\mu_{g}} \leq \frac{\int_{B_{R}}f^{2}d\mu_{g}}{\int_{B_{R}}Hf^{2}d\mu_{g}} \leq \frac{\int_{B_{R}}f^{2}d\mu_{g}}{\int_{B_{R}}H_{k}(r)f^{2}d\mu_{g}}.$$
We want to study the ratios above as $R\to \infty$. To this aim, we apply de l'H\^{o}pital's rule as the numerator and denominator in each term go to infinity:
$$ \frac{1}{(n-1)\sqrt{|k^\prime|}} \leq \lim_{R \to \infty}\frac{\int_{S_{R}}f^{2}d\sigma_{R}}{\int_{S_{R}}Hf^{2}d\sigma_{R}} \leq \frac{1}{(n-1)\sqrt{|k|}}.$$
Hence, by employing the bound in Equation \eqref{doubleHf}, we retrieve the \emph{exponential bounds} for $\int_{B_{R}}f^{2}d\mu_{g}$ stated in the theorem with some constants $C_1$ and $C_2$.
\end{proof}

\section{On the Singer Conjecture in Degree One}\label{contradiction}

The Singer conjecture predicts the vanishing of the $L^2$-Betti numbers of non-positively curved closed Riemannian manifolds outside the middle dimension. We refer to the recent book \cite{Maxim} for many different points of view on this problem of current research interest. The Singer conjecture is not known in general even in degree one for negatively curved spaces, see \cite[Section 7]{DS22}. In degree one, the vanishing of the first $L^2$-Betti number implies the non-existence of $L^2$-integrable harmonic one forms on the Riemannian universal cover of closed non-positively curved Riemannian manifolds. This is equivalent to the non-existence of harmonic functions with $L^2$-integrable gradient. For this reason, it is of interest to look at the properties of harmonic functions with $L^2$-integrable gradient on non-positively curved simply connected Riemannian manifolds. Our main estimate (Theorem \ref{doublebound}) can be used to prove that such creatures have \textit{small} $L^2$-energy. This is the content of the following result.

\begin{theorem}\label{thm:11.1}
   Let $(M,g)$ be complete and simply connected $n$-manifold with sectional curvatures $k' \leq \text{sec}_{g} \leq k < 0$. Let $f: M \to \IR$ be non-constant and harmonic. Suppose $\int_{B_{R}}|\nabla f|^{2} d\mu_{g}$ is bounded by a uniform constant for every $R$. Then, we have the following inequality for positive constants $C$ and $D$:
$$C\text{Vol}_{g_{h}^{k}}(R)\leq \int_{B_{R}}f^{2}d\mu_{g} \leq D\text{Vol}_{g_{h}^{k'}}(R),$$
where $\text{Vol}_{g_{h}^{k}}(R)$ and $\text{Vol}_{g_{h}^{k'}}(R)$ denote the volumes of a geodesic sphere of radius $R$ in the hyperbolic $n$-space with sectional curvature $k$ and $k^\prime$ respectively.
\end{theorem}

\begin{proof}

For $R>1$, Theorem \ref{doublebound} gives us 
$$C_1e^{\int_{1}^{R}\frac{H_{k}(s)}{1-\mu(s)}ds} \leq \int_{B_{R}}f^{2}d\mu_{g} \leq C_2e^{\int_{1}^{R}\frac{H_{k'}(s)}{1-\mu(s)}ds}.$$
We may minimize the lower bound by setting $\mu(s) = 0$. With this assumption, direct computation yields
$$C\text{Vol}_{g_{h}^{k}}(R) \leq \int_{B_{R}}f^{2}d\mu_{g}.$$
This is the desired lower bound. The remainder of the proof focuses on developing the upper bound. Suppose $\int_{B_{R}}|\nabla f|^{2}d\mu_{g} \leq \sigma$ for any $R>0$. Since $f^2$ is non-negative and sub-harmonic, the mean-value property ensures that
\[
\int_{S_{R}}f^2d\sigma_R\geq c\sinh^{n-1}(\sqrt{|k|}R),
\]
for some constant $c>0$. Thus, by Definition \ref{defnmu} we have the following comparison
$$\mu(R) \leq \frac{2\sigma R}{c\sinh^{n-1}(\sqrt{|k|}R)}.$$
For the sake of simplicity, we allow $\sigma$ to absorb other constant coefficients in its term. Then, we have the following estimate:
\begin{equation}
\label{eq:11.3}
 \int_{B_{R}}f^{2}d\mu_{g} \leq C_2e^{(n-1)\sqrt{|k'|}\int_{1}^{R}\frac{\coth(\sqrt{|k'|}s)}{1-\frac{\sigma s}{\sinh^{n-1}(\sqrt{|k|}s)}}ds}
\end{equation}
We aim to integrate $\int_{1}^{R}\frac{\coth(\sqrt{|k'|}s)}{1-\frac{\sigma s}{\sinh^{n-1}(\sqrt{|k|}s)}}ds$ that we conveniently re-write as:
$$\int_{1}^{R}\frac{\coth(\sqrt{|k'|}s)\sinh^{n-1}(\sqrt{|k|}s)}{\sinh^{n-1}(\sqrt{|k|}s)-\sigma s}ds.$$
We proceed by integration by parts, where we shall integrate $\coth$ and derive the rational function that remains
\[
u(s):=\frac{\sinh^{n-1}(\sqrt{|k|}s)}{\sinh^{n-1}(\sqrt{|k|s}) - \sigma s}.
\]
The antiderivative of $\coth(\sqrt{|k'|}s)$ is $\frac{1}{\sqrt{|k'|}}\text{ln}|\sinh(\sqrt{|k'|s})|$. Direct computation yields
$$u'(s) = -\frac{\sigma \sinh^{n-2}(\sqrt{|k|}s)[\sqrt{|k|}(n-1)s\cosh(\sqrt{|k|}s) - \sinh(\sqrt{|k|}s)]}{(\sinh^{n-1}(\sqrt{|k|}s) - \sigma s)^{2}}.$$
Now, we express the integration by parts as:
\begin{align} \label{eq:11.4}
    \int_{1}^{R}\frac{\coth(\sqrt{|k'|}s)\sinh^{n-1}(\sqrt{|k|}s)}{\sinh^{n-1}(\sqrt{|k|}s)-\sigma s}ds &= \left[\frac{1}{\sqrt{|k'|}}\text{ln}|\sinh(\sqrt{|k'|s})|u(s)\right]_{1}^{R} \\ \notag
& - \int_{1}^{R}\coth(\sqrt{|k'|}s)u'(s)ds.
\end{align}
We aim to demonstrate that $- \int_{1}^{R}\coth(\sqrt{|k'|}s)u'(s)ds$ is bounded above. To do this, we primarily focus on $-u'(s)$ (we pull an extra negative to cancel with the one outside the integral). 
Taking the limit as $s \to \infty$, we have
$$\lim_{s \to \infty}u'(s) = \lim_{s \to \infty}\frac{s}{e^{(n-1)\sqrt{|k|}s}} = 0.$$
That is to say, $u'(s)$ goes to zero exponentially. Taken together with the observation that $\coth$ approaches 1, it is easy to conclude that the second term in Equation \eqref{eq:11.4} is bounded. Now, we turn our attention to the other term in Equation \eqref{eq:11.4}:
$$\frac{1}{\sqrt{|k'|}}\ln\left|\sinh\left(\sqrt{|k'|}R\right)\right|u(R).$$
Passing this directly to the power over the exponential as in Equation \eqref{eq:11.3}, we get
$$e^{\frac{(n-1)\sqrt{|k'|}}{\sqrt{|k'|}}\ln|\sinh(\sqrt{|k'|}R)|u(R)} = \sinh^{(n-1)u(R)}(\sqrt{|k'|}R).$$
We perform a limit comparison of this function with volume:
$$\lim_{R\to \infty} \frac{\sinh^{(n-1)u(R)}(\sqrt{|k'|}R)}{\sinh^{n-1}(\sqrt{|k'|}R)} = \lim_{R \to \infty}\frac{e^{\sqrt{|k'|}(n-1)u(R)R}}{e^{\sqrt{|k'|}(n-1)R}} = \lim_{R \to \infty}e^{\sqrt{|k'|}(n-1)(u(R) - 1)R}.$$
It is clear that $\lim_{R \to \infty}u(R) = 1$. So, $\lim_{R \to \infty}(u(R) - 1)R = 0 \cdot \infty$. 
We apply de l'H\^opital's rule:
$$\lim_{R \to \infty}(u(R) - 1)R = \lim_{R \to \infty}\frac{(u(R) - 1)}{\left(\frac{1}{R}\right)} = \lim_{R\to \infty}-R^{2}u'(R).$$
We know that $u'(R) = O(Re^{-(n-1)\sqrt{|k'|}R})$. So, $\lim_{R \to \infty}(u(R) - 1)R = 0$. Hence, we have that the function 
\[
e^{(n-1)\sqrt{|k'|}\int_{1}^{R}\frac{\coth(\sqrt{|k'|}s)}{1-\frac{\sigma s}{\sinh^{n-1}(\sqrt{|k|}s)}}ds}
\]
grows like $\sinh^{n-1}(\sqrt{|k'|}R)$. This is exactly $\text{Vol}_{g_{h}^{k'}}(R)$ up to constant. The proof is complete.
\end{proof}

\section{Estimates for Positive Harmonic Functions}\label{PH}

In this section, we prove that the hyperbolic Poisson kernel is, in a sense, the positive hyperbolic harmonic function with the fastest growing $L^2$-norm over geodesic balls. This gives us an upper bound on the growth of \emph{all} positive hyperbolic harmonic functions. We use this fact to prove a similar result for positive harmonic functions on simply connected negatively curved Riemannian manifolds with pinched sectional curvature. We start with the following result for positive harmonic function on the real-hyperbolic space.

\begin{proposition} \label{positiveh}
 Let $(\mathbb{H}^n_{\IR},g^{-1}_h)$ be $n$-dimensional hyperbolic space with sectional curvature equal to $-1$. Given a  positive harmonic function $f: \mathbb{H}^n_{\IR} \to \IR$, we have the following bounds:
$$C_2e^{(n-1)R} \leq \int_{B_{R}}f^{2}d\mu_{g} \leq D_2e^{2(n-1)R},$$
for some positive constants $C_2$, $D_2$ (depending on $f$).
\end{proposition}
\begin{proof}
The lower bounds follows directly from Theorem \ref{doublebound} and the fact that the volume of geodesic balls in the $n$-hyperbolic space with sectional curvature $-1$ grows like $e^{(n-1)R}$. For the upper bound, we begin by comparing the $L^2$-norm of a positive harmonic function $f$ with the $L^{2}$-norm of the hyperbolic Poisson kernel $P_{-1}$. Recall that $P_{-1}$ is a positive harmonic function satisfying the pointwise equality
\[
|\nabla P_{-1}|^2=(n-1)^2 P_{-1}^2.
\]
Our starting point is the divergence identity given in Equation \eqref{eq:4.1}:
$$\int_{S_{R}}f^{2}d\sigma_{R} = \int_{B_{R}}H_{-1}(r)f^{2}d\mu_{g^{-1}_h} + \int_{B_{R}}\partial_{r}f^2 d\mu_{g^{-1}_h},$$
and its equivalent formulation in terms of Green's identity:
$$\int_{S_{R}}f^{2}d\sigma_{R} = \int_{B_{R}}H_{-1}(r)f^{2}d\mu_{g^{-1}_h} + 2\int_{0}^{R}\int_{B_{r}}|\nabla f|^{2} d\mu_{g^{-1}_h}dr.$$
We take one derivative:
$$\frac{\partial}{\partial R}\int_{S_{R}}f^{2}d\sigma_{R} = H_{-1}(R)\int_{S_{R}}f^{2}d\sigma_{R} + 2\int_{B_{R}}|\nabla f|^{2} d\mu_{g^{-1}_h}.$$
Now, we apply  to $f$ the Cheng-Yau gradient estimate (\cite[Theorem 1.1]{PLi} or \cite[Chapter I, \S 3]{SchoenY}) for positive harmonic functions on spaces with Ricci curvature bounded below by $-(n-1)$:
\[
|\nabla f|^2\leq (n-1)^2 f^2.
\]
This pointwise inequality implies:
$$\frac{\partial}{\partial R}\int_{S_{R}}f^{2}d\sigma_{R} \leq H_{-1}(R)\int_{S_{R}}f^{2}d\sigma_{R} + 2(n-1)^{2}\int_{B_{R}}f^{2} d\mu_{g^{-1}_h}.$$
Thus, we adopt the notation $y(R) = \int_{S_{R}}f^{2}d\sigma_{R}$, and we obtain the following differential inequality:
\begin{equation}
\label{eq1}
y'(R) \leq H_{-1}(R)y(R) + 2(n-1)^{2}\int_{0}^{R}y(r) dr.
\end{equation}
We apply the same steps for $P_{-1}$, noting that Li-Yau's gradient estimate is now saturated by $P_{-1}$.  By adopting the notation $Q(R) = \int_{S_{R}}P_{-1}^{2}d\sigma_{R}$, we obtain:
\begin{equation}
\label{eq2}
Q'(R) = H_{-1}(R)Q(R) + 2(n-1)^{2}\int_{0}^{R}Q(r) dr.
\end{equation}
Now, for some small constant $c>0$, there exists $\varepsilon > 0$ such that $Q(R) - cy(R) > 0$ for any $R\in [0, \varepsilon)$. We want to extend this result to all $R \in (0,\infty)$. We begin by subtracting Equation \eqref{eq1} scaled by $c$ from Equation \eqref{eq2}:
\begin{equation}\label{dineq}
(Q(R) - cy(R))' \geq H_{-1}(R)[Q(R) - cy(R)] + 2(n-1)^{2}\int_{0}^{R}[Q(r) - cy(r)] dr.
\end{equation}
Suppose for the sake of contradiction that for some $R_{0} > \varepsilon$, $Q(R_{0}) < cy(R_{0})$. Let $R_{0}$ be the first value where this happens. Since $Q(0)-cy(0)>0$,  the derivative $(Q(R) - cy(R))'$ has to turn negative at some previous time.  On the other hand, this is impossible because of Equation \eqref{dineq}. We conclude that $(Q(R) - cy(R))'\geq 0$ for any $R>0$. Thus, $y(R)$ is bounded above (up to a positive constant) by $Q(R)$. In Appendix \ref{Hayden}, we compute the integral of the hyperbolic Poisson kernel over a geodesic sphere in terms of hypergeometric functions. Among other things, this exact integration implies the bound $e^{2(n-1)R}$ stated in the theorem. For more details, see Theorem \ref{hypergeometric} in Appendix \ref{Hayden}.
\end{proof}

Using the Price inequality and the ideas developed in Proposition \ref{positiveh}, we may finally provide bounds on the asymptotic growth of the $L^{2}$-norm of positive harmonic functions on simply connected Riemannian manifolds with pinched negative sectional curvature.

\begin{theorem}\label{PHmain}
Let $(M^{n},g)$ be complete and simply connected with sectional curvatures $k' \leq \text{sec}_{g} \leq k < 0$. Let $f: M \to \IR$ be a positive harmonic functions. Then, we have the following bounds:
$$B_1e^{\sqrt{|k|}(n-1)R} \leq \int_{B_{R}}f^{2}d\mu_{g} \leq B_2e^{2\sqrt{|k'|}(n-1)R},$$
for some positive constants $B_1$ and $B_2$.
\end{theorem}
\begin{proof}
The lower bounds follows directly from Theorem \ref{doublebound} and the fact that the volume of geodesic balls in the $n$-hyperbolic space with sectional curvature $k$ grows like $e^{|k|(n-1)R}$. Therefore, it remains to provide the upper bound. For this, we argue as in Proposition \ref{positiveh}. We begin by taking a derivative of the usual divergence equality:
$$\frac{\partial}{\partial R}\int_{S_{R}}f^{2}d\sigma_{R} = \int_{S_{R}}H(R, \theta)f^{2}d\sigma_{R} + 2\int_{B_{R}}|\nabla f|^{2} d\mu_{g},$$
where
\[
H(R, \theta): S_R\longrightarrow \IR
\]
is the (variable) mean-curvature of the geodesic sphere $S_R$ in $(M^n, g)$. Next, we observe that
\[
k' \leq \text{sec}_{g} \leq k < 0 \quad \Rightarrow \quad Ric_g\geq -(n-1)|k^\prime|g.
\]
We the apply  to $f$ the Cheng-Yau gradient estimate (\cite[Theorem 1.1]{PLi} or \cite[Chapter I, \S 3]{SchoenY}) for positive harmonic functions on spaces with Ricci curvature bounded below by $-|k^\prime|(n-1)$:
\[
|\nabla f|^2\leq (n-1)^2 |k^\prime|f^2.
\]
This pointwise inequality when combined with Rauch's estimate
\[
H_{k^\prime}(R)\geq  H(R, \theta)
\]
implies:
$$\frac{\partial}{\partial R}\int_{S_{R}}f^{2}d\sigma_{R} \leq H_{k^\prime}(R)\int_{S_{R}}f^{2}d\sigma_{R} + 2(n-1)^{2}|k^\prime|\int_{B_{R}}f^{2} d\mu_{g}.$$
Thus, we adopt the notation $y(R) = \int_{S_{R}}f^{2}d\sigma_{R}$, and we obtain the following differential inequality:
\begin{equation}
\label{eq:17.6}
y'(R) \leq H_{k^\prime}(R)y(R) + 2(n-1)^{2}|k^\prime|\int_{0}^{R}y(r) dr.
\end{equation}
We denote by $P_{k^\prime}$ a Poisson kernel in real-hyperbolic space with sectional curvature $k^\prime$. This is a positive harmonic function satisfying the pointwise equality:
\[
|\nabla P_{k^\prime}|^2=(n-1)^2|k^\prime| P_{k^\prime}^2.
\]
By adopting the notation $Q(R) = \int_{S_{R}}P^{2}_{k^\prime}d\sigma_{R}$, we obtain:
\begin{equation}
\label{eq:17.7}
Q'(R) = H_{k^\prime}(R)Q(R) + 2(n-1)^{2}|k^\prime|\int_{0}^{R}Q(r) dr.
\end{equation}
Now, for some small constant $c>0$, there exists $\varepsilon > 0$ such that $Q(R) - cy(R) > 0$ for any $R\in [0, \varepsilon)$. We want to extend this result to all $R \in (0,\infty)$. We begin by subtracting Equation \eqref{eq:17.6} scaled by $c$ from Equation \eqref{eq:17.7}:
\begin{equation}\label{diffineq}
(Q(R) - cy(R))' \geq H_{k^\prime}(R)[Q(R) - cy(R)] + 2(n-1)^{2}|k^\prime|\int_{0}^{R}[Q(r) - cy(r)] dr.
\end{equation}
Arguing as in Proposition \ref{positiveh}, we have that that $y(R)$ is bounded above (up to a constant) by $Q(R)$. Now, with the sectional curvature normalized to be $-k^\prime$, $Q(R)$ grows like $e^{2\sqrt{|k'|}(n-1)R}$, see again Appendix \ref{Hayden}. This gives the upper bound in the theorem and the proof is complete.
\end{proof}

\appendix

\section{Appendix A: $\mu$-Function on Euclidean Space}\label{Aaron}

In this appendix, we study the $\mu$-function (see Definition \ref{defnmu}) on Euclidean space and highlight its similarities with Almgren's frequency function.

Recall that any harmonic function on $\IR^n$ decomposes as a series of homogeneous harmonic polynomials, see for example \cite[Corollary 5.34]{Axler}. The $\mu$-function is constant on spherical harmonics.
\begin{lemma}
Let $v_d$ be a homogeneous harmonic polynomial of degree $d$ on $\IR^n$. We have
\[
\mu_{v_d}(R)=\mu_d:=\frac{2d}{2d+n-1},
\]
for any $R\geq 0$.
\end{lemma}
\begin{proof}
A direct computation yields
\[
\int_{S_R}v^2_d d\sigma=\epsilon_d R^{2d+n-1}
\]
where $\epsilon_d:=\int_{S_1}v^2_d d\sigma$. Similarly, one has
\[
\int_{B_R} |\nabla v_d|^2d\mu = d\epsilon_d R^{2d+n-2}.
\]
Therefore, we have
\[
\mu_{v_d}(R)=\frac{2\int^R_0\int_{B_r}|\nabla v_d|^2d\mu dr}{\int_{S_R}v^2_d d\sigma}=\frac{2d\epsilon_d\int^R_0 r^{2d+n-2}dr}{\epsilon_d R^{2d+n-1}}=\frac{2d}{2d+n-1}=\mu_d.
\]
\end{proof}

We can now state our main \textit{monotonicity} result.

\begin{theorem}\label{MM}
If $u:\IR^n\to\IR$ be harmonic, then $\mu_{u}(R)$ is strictly monotonic increasing unless $u$ is a homogeneous harmonic polynomial.  
\end{theorem}
\begin{proof}
First, we decompose $u$ as a series of harmonic homogeneous polynomials
\[
u=\sum_{j} v_{d_j}.
\]
This series decomposition converges absolutely and uniformly on compact subsets of $\IR^n$. Now, a somewhat tedious computation yields:
\[
\mu^{\prime}(R)=\frac{\sum^{\infty}_{j=1}\sum_{k>j}2(d_{k}-d_{j})(\mu_{d_k}-\mu_{d_j})\epsilon_{d_k}\epsilon_{d_j}R^{2(d_k+d_j-2d_1)-1}}{\sum^{\infty}_{j=1}\sum_{k>j}2\epsilon_{d_j}\epsilon_{d_k}R^{2(d_j+d_k-2d_1)}+\sum^{\infty}_{j=1}\epsilon^2_{d_j}R^{4(d_j-d_1)}}
\]
As $d_j<d_k$ for $j<k$, we have $\mu_{d_j}<\mu_{d_k}$. Thus, as long as $u$ is not a homogeneous harmonic polynomial, we have $\mu^{\prime}(R)>0$. For more analytical details, we refer to the PhD thesis of the third named author \cite[Theorem 3.6.1]{Aaron}.
\end{proof}
 
Now, we compare the behavior of $\mu$ with Almegren's frequency function $U$. Recall, that given a Euclidean harmonic function $u:\IR^n\to\IR$, Almgren defined the useful function
\[
U(R)=\frac{R\int_{B_R}|\nabla u|^2d\mu}{\int_{S_{R}}u^2d\sigma}.
\]
It is well-known that  $U$ is strictly monotonic increasing unless $u$ is a homogeneous harmonic polynomial. In this case, $U$ is constant and equals the degree of the homogeneous harmonic polynomial. For these results, see for example \cite[Lemma 6.3 and Lemma 6.4]{Colding}. Thus, the $\mu$-function exhibits analogous monotonicity behavior as to Almgren's function, \textit{cf}. Theorem \ref{MM}. The $\mu$-function has the appearance of a ``dampened'' Almgren's function as it is always bounded between zero and one. Moreover, one can prove that $U$ is unbounded if and only if $u$ has infinitely many harmonic homogeneous polynomials appearing in its series decomposition. Interestingly, this is the case if and only if the $\lim_{R\to\infty}\mu(R)=1$. We refer to \cite[Chapter 3]{Aaron} for an extensive analytical and numerical comparison between $\mu$ and $U$.
\section{Appendix B: On the Integral of the Hyperbolic Poisson Kernel over a Geodesic Sphere}\label{Hayden}
In this appendix, we estimate the integral of the hyperbolic Poisson kernel $P_{-1}$ on a geodesic sphere in real-hyperbolic space $(\mathbb{H}^n_{\IR},g^{-1}_h)$.
\begin{proposition}\label{hypergeometric}
There exists a constant $c>0$ such that
\[
\int_{S_{R}}P^2_{-1}d\sigma_{R} \leq c \sinh^{n-1}(R)e^{(n-1)R}.
\]
\end{proposition}
\begin{proof}
By adopting the notation $Q(R)=\int_{S_{R}}P^2_{-1}d\sigma_{R}$ as in Section \ref{PH}, we obtain that $Q(R)$ satisfies the second order linear differential equation
\[
Q^{\prime\prime}-(n-1)\coth(R)Q^{\prime}+[(n-1)\csch^2(R)-2(n-1)^2]Q=0.
\]
 This equation follows directly by taking one derivative of Equation \eqref{eq2} and recalling that $H(R)=(n-1)\coth(R)$. Next, we define an auxiliary function $u(R)$ by setting
\[
Q(R)=\sinh^{n-1}(R)u(R).
\]
Since the area of a geodesic sphere of radius $R$ in $(\mathbb{H}^n_{\IR},g^{-1}_h)$ is proportional to $\sinh^{n-1}(R)$, we notice that $u(0)\neq 0$ as $P_{-1}$ is nowhere vanishing. Next, one computes that $u(R)$ satisfies the second order linear differential equation
\begin{equation}\label{L}
u^{\prime\prime}+(n-1)\coth(R)u^{\prime}-2(n-1)^2u=0.
\end{equation}
Notice that for any constant $c>0$, the function $ce^{(n-1)R}$ is a subsolution for the second order linear differential operator associated to Equation \eqref{L}
\[
L=\frac{d^2}{dR^2}+(n-1)\coth(R)\frac{d}{dR}-2(n-1)^2.
\]
Since $-2(n-1)^2<0$, the standard maximum principle argument tells us that there exists $c>0$ such that
\[
u(R)\leq ce^{(n-1)R}
\]
for any $R\geq 0$. This completes the proof.
\end{proof}
Concluding, since $Q(R)=\int_{S_{R}}P^2_{-1}d\sigma_{R}$ as in Section \ref{PH}, we obtain that $Q(R)$ grows at most like $e^{2(n-1)R}$ as used in the proof of Proposition \ref{positiveh}. By rescaling the hyperbolic metric to have sectional curvature $-k^{\prime}$, we obtain the corresponding growth rate $e^{2\sqrt{|k^{\prime}|}(n-1)R}$ as used at the end of the proof of Theorem \ref{PHmain}.


\end{document}